\documentclass{article}
\usepackage[margin=4cm]{geometry}
\usepackage{amsmath,amssymb,amsthm,hyperref,comment}
\hypersetup{
    colorlinks=true,
    linkcolor=blue,
    filecolor=magenta,      
    urlcolor=cyan,
}

\newtheorem{lemma}{Lemma}[section]
\newtheorem{definition}{Definition}[section]
\newtheorem{theorem}{Theorem}[section]
\newtheorem{proposition}{Proposition}[section]
\newtheorem{corollary}{Corollary}[section]

\newcommand{\ddbar}{\sqrt{-1}\partial\bar{\partial}}
\newcommand{\polydisc}{\frac{3}{4}\mathcal{P}_k}
\newcommand{\remark}{\noindent\emph{Remark. }}

\title{The complex Monge-Amp\`{e}re equation on the complement of a divisor}
\author{Fangyu Zou}

\begin{document}
\maketitle
\begin{abstract}
We consider the complex Monge-Amp\`{e}re equation on complete K\"{a}hler manifolds with cusp singularity along a divisor when the right hand side $F$ has rather weak regularity. We proved that when the right hand side $F$ is in some weighted $W^{1,p_0}$ space for $p_0>2n$, the Monge-Amp\'{e}re equation has a classical $W^{3,p_0}$ solution. 
\end{abstract}

\section{Introduction}

Let $(\bar{M},\omega_0)$ be a compact K\"{a}hler manifold of complex dimension $n$, in which we consider a divisor $D$ with only \emph{simple normal crossings} with decomposition $D=\sum_{j=1}^N D_j$ into smooth irreducible components. Let $[D_j]$ be the associated line bundle to each $D_j$, endowed with a smooth hermitian metric $|\cdot|_j$, and $\sigma_j\in \mathcal{O}([D_j])$ be a holomorphic defining section such that $D_j=\{\sigma_j=0\}$ for each $j$. Let $\rho_j=-\log(|\sigma_j|^2_j)$. Up to multiplying $|\cdot|_j$ by a positive constant or a smooth positive function, we can assume that $|\sigma_j|_j^2\leq e^{-1}$ so that $\rho_j\geq 1$ out of $D_j$. Note that $\ddbar\rho_j$ extends to a \emph{smooth} real $(1,1)$-form on the whole $\bar{M}$, which lies in the class $2\pi c_1([D_j])$. Let $\rho=\prod_{j=1}^N \rho_j$. Let $\lambda>0$ be a real parameter sufficiently large. Set
\begin{equation*}
\omega = \lambda \omega_0 + \ddbar\log{\rho}
= \lambda \omega_0 + \sum_{j=1}^N \ddbar\log(-\log|\sigma_j|_j^2), 
\end{equation*}
then $\omega$ defines a genuine K\"{a}hler form on $M=\bar{M}-D$. It has properties that it is complete, with finite volume, with cusp singularity along $D$ and has injectivity radius going to zero. 

The purpose of this paper is to study the Monge-Amp\'{e}re equation 
\begin{equation}\label{equ14}
(\omega + \ddbar{\phi})^n = e^F\omega^n, \,\  \int_M (e^F-1) dV = 0
\end{equation}
when the right hand side $F$ has rather weak regularity. We show that when $F$ is in some \emph{weighted} $W^{1,p_0}$ space for $p_0> 2n$, more precisely, when 
\begin{equation}
I(F,p_0):=\int_M (|F|^{p_0}+|\nabla F|^{p_0})\rho^{\frac{p_0-2}{2n-2}}dV < \infty,
\end{equation}
the equation (\ref{equ14}) has a classical solution $\phi$ in $W^{3,p_0}$ (not in the weighted sense). In particular, we proved

\begin{theorem}\label{thm1}
Let $\bar{M}$ be a compact K\"{a}hler manifold of complex dimension $n$ and $D$ be an divisor on $\bar{M}$ with only simple normal crossings. Let $M=\bar{M}-D$ and $\omega$ be the reference K\"{a}hler metric on $M$ constructed above. For any function $f$ satisfying $I(F,p_0)< \infty$ for some $p_0> 2n$, the Monge-Amper\'{e} equation (\ref{equ14}) has a classical solution $\phi$ in $W^{3,p_0}(M)$.
\end{theorem}

This regularity problem in the compact K\"{a}hler manifold setting has been proved by X. Chen and W. He in \cite{Chen-He}. They proved that if $(M,\omega)$ is a compact K\"{a}hler manifold, then the equation (\ref{equ14}) has a classical solution in $W^{3,p_0}$ once the right hand side $F$ is in $W^{1,p_0}$ for some $p_0>2n$. In Yau's seminal resolution of the Calabi conjecture \cite{Yau1}, the maximum principle is used in a significant way in the $C^2$ estimate of the potential function, where the proof depends heavily on the $C^0$ norm of $\Delta F$. Since we  have weaker regularity on $F$, the main issue is that we can not apply maximum principle to get $C^2$ estimate. One main innovation of \cite{Chen-He} is that they derive the gradient estimate and Laplacian estimate by integration methods. The main tool is the Moser's iteration technique (see \cite{Moser}).

Our result can be viewed as the non-compact version of Chen-He's theorem. We will follow their idea of seeking an integration method for the gradient and Laplacian estimates. Where our case differs from the compact case is two folds: first, we need consider the boundary term when we do integration by parts; second, the usual Sobolev inequality fails in our context.


To overcome the issue, we consider the following the $\epsilon$-perturbed equation of (\ref{equ14}):
\begin{equation}\label{equ1}
(\omega+\ddbar{\phi_\epsilon})^n= e^{F+\epsilon\phi_\epsilon}\omega^n,\,\, \int_M (e^F-1)dV = 0
\end{equation}
for $\epsilon\in (0,1]$. Equation (\ref{equ1}) has been well studied by Cheng-Yau \cite{Cheng-Yau}, R. Kobayashi \cite{Kobayashi} and Tian-Yau \cite{Tian-Yau} to derive K\"{a}hler-Einstein metrics with negative curvature on $(M,\omega)$ assuming $K_{\bar{M}}+D$ ample. However, the existence of solution $\phi_\epsilon$ of (\ref{equ1}) is proved by the continuity method in quasi-coordinates without additional assumption of the ampleness of $K_{\bar{M}}+D$. 

The idea here is that assuming $F$ is smooth and compactly supported, by maximum principle we can first derive the a priori estimates of $\phi_\epsilon$ depending on the parameter $\epsilon$ and norms of higher derivatives of $F$. This garantees the integrand of boundary terms in the integration by parts are at least in $L^1$. Then we apply a theorem of Gaffney-Stokes \cite{Gaffney-Stokes} to show the boundary terms vanish. Hence, we can do integration by parts as in the compact case. To emphasize, this is just a technical step to make the integration by parts work through. The final estimate will be uniform in $\epsilon$ and only depends on $I(F,q_0)$.

On the other hand, to deal with the problem of lack of Sobolev inequlity, we borrow the weighted Sobolev inequality developed in \cite{AUV} instead. There is still ssome serious issue need to deal with as a result of that the measure (the volume form multiplied by the weight) is not finite. A key obervation is that one higher order term in the Chen-He's inequalites dominates the constants and make the Moser's iteration possible in our case, while in the compact case this high order term can be directly dropped off.

The main idea of this paper is to derive a uniform $W^{3,p_0}$ estimate of the solutions $\phi_\epsilon$ of the $\epsilon$-perturbed equation, following by taking a converging subsequence. We follow Chen-He's routine to first derive uniform $C^1$ and $C^2$ estimates. In this process, we can assume that $F\in C^\infty_c(M)$ since we can approximate $F$ by $C^\infty_c(M)$ funtions in the weighted Sobolev spaces when $I(F,p_0)< \infty$ (see Lemma \ref{app_lem} in section 6). In particular, we have
 
\begin{theorem}\label{thm2}
Suppose $F\in C^\infty_c(M)$ satisfies $I(F,p_0)<\infty$ for some $p_0>2n$. If $\phi_\epsilon$ is a solution to the perturbed equation (\ref{equ1}), then there exists a constant $C=C(I(F,p_0),p_0,m,\omega)$ such that for all $\epsilon\in (0,1]$,
\begin{equation}
|\nabla \phi_\epsilon|\leq C.
\end{equation}
\end{theorem}

\begin{theorem}\label{thm3}
Suppose $F\in C^\infty_c(M)$ satisfies the conditions in Theorem \ref{thm1}. If $\phi_\epsilon$ is a solution to the perturbed (\ref{equ1}), then there exists a constant $C=C(I(F,p_0),p_0,m,\omega)$ such that for all $\epsilon\in (0,1]$,
\begin{equation}
|\Delta \phi_\epsilon|\leq C.
\end{equation}
\end{theorem}

The arrangement of this paper is the following: In section 2 we set up the ingredients for the proof of main theorems; in section 3 we cite a proof of Auvray on uniform $C^0$ estimate with slight modification; section 4 and 5 are devoted to the proof of Theorem \ref{thm2} and Theorem \ref{thm3}; in section 4 we first derive a $C^{2,\alpha}$ estimate in quasi-coordiates and finally the $W^{3,p_0}$ estimate. The main result Theorem \ref{thm1} is proved in the end of the paper. 
\vspace{1em}

\noindent \textbf{Convention of notations}: With a little abuse of notation, we will use the K\"{a}hler form $\omega$ to denote the reference metric, while in some cases it is also denoted as $g$. Throughout this paper, $dV,\, \nabla$ and $\Delta$ denote the volume form, the Levi-Civita connection and the Laplacian operator of the reference metric $\omega$, respectively; $dV',\, \nabla'$ and $\Delta'$ will be those of the metric $\omega_{\phi_\epsilon}=\omega+\ddbar{\phi_\epsilon}$. For simplicity of notations, we will also drop the subscript $\epsilon$ from $\phi_\epsilon$ when this is no ambiguity. The constant ``$C$" without subscript may vary from line to line, while if there is subscript, then it is some fixed constant.  Constants in this paper will only depend on $(I(F,p_0),p_0,n,\omega)$ unless specifically pointed out.
\vspace{1em}

\textbf{Acknowledgement:} {The author is very grateful to his advisor Prof. Xiuxiong Chen for introducing this problem  and for his consistent support during the proof. The author is also thankful to Yuanqi Wang, Gao Chen, Ruijie Yang and Santai Qu for useful discussions.}

\section{Preliminaries}

In this section we set up the ingredients for the proof of the main theorems. 

\subsection{The reference metric}

Let us quickly recall the construction of the reference metric $\omega$. Let $(\bar{M},\omega_0)$ be a compact K\"{a}hler manifold of complex dimension $n$, in which we consider a divisor $D$ with simple normal crossings with decomposition $D=\sum_{j=1}^N D_j$ into smooth irreducible components. For each $j$, let $\sigma_j$ be a holomorphic defining section of $D_j$. We can assume that $\rho_j:=-\log(|\sigma_j|^2)\geq 1$ out of $D_j$. Note that $\ddbar{\rho_j}$ extends to a smooth real $(1,1)$-form on the whole $\bar{M}$, whose class is $2\pi c_1([D_j])$. Let $\rho=\prod_{j=1}^N \rho_j$. For some postive real parameter $\lambda>0$, set 

\begin{equation}
\omega = \lambda \omega_0 - \ddbar\log{\rho}
= \lambda \omega_0 - \sum_{j=1}^N \ddbar\log(-\log|\sigma_j|^2).
\end{equation}

\begin{lemma}
For $\lambda>0$ sufficiently large, $\omega$ defines a K\"{a}hler metric on $M=\bar{M}-D$.
\end{lemma}
\begin{proof}
This follows from a simple computation. Note that
\[
-\ddbar{\log\rho_j} = \frac{\sqrt{-1}\partial \rho_j \wedge \bar{\partial} \rho_j}{\rho_j^2} -
 \frac{ \ddbar \rho_j}{\rho_j}.
\]
The first summand is a postive $(1,1)$-form. For each $j$, there is some positive $\lambda_j>0$ such that $\ddbar{\rho_j} \leq \lambda_j \omega_0$ on $\bar{M}$. Hence, $\lambda_j\omega_0+\ddbar{\log \rho_j}>0$ on $\bar{M}-D_j$. Let $\lambda = \sum_j \lambda_j$, then 
\[
\omega = \lambda \omega_0 -\ddbar{\log\rho}=\sum_j (\lambda_j \omega_0 - \ddbar{\log \rho_j}) > 0
\]
on $M=\bar{M}-D$.
\end{proof}

A simple model for this type of metric is the punctured disc $\Delta^*=\Delta-\{0\}$ with the Poincar\'{e} metric $\omega = -\ddbar{\log(-\log|z|^2)} = \frac{\sqrt{-1}dz\wedge d\bar{z}}{|z|^2\log^2|z|^2}$. Indeed, the asymptotics of the reference metric near $D$ can be compared with this kind of model metric. 

Soppose $P$ is a point in a crossing of codimension $k$, say, $D_1\cap\cdots\cap D_k$. Take an open neighborhood $U$ of $P$ which is biholomorphic to coordinate chart $(\Delta^n; z_1,\ldots,z_n)$. The simple normal crossing assumption allows us to write $D\cap U = (D_1\cup\cdots\cup D_k)\cap U=\cup_{i=1}^k\{z_i=0\}$, where $z_j=0$ being the equation of $D_j$ in $U$. Then $U\backslash D=(\Delta^*)^k\times \Delta^{n-k}$. Let $\omega_{mdl}$ be the model metric on $(\Delta^*)^{k}\times \Delta^{n-k}$:
\[
\omega_{mdl}= \sum_{j=1}^k \frac{\sqrt{-1}dz_j\wedge d\bar{z}_j}{|z_j|\log^2|z_j|^2}
+ \sum_{j=k+1}^n \sqrt{-1}dz_j \wedge d\bar{z}_j.
\]

\begin{lemma}\label{lemma2}
In the coordinates $(z_1,\ldots,z_k,z_{k+1},\ldots,z_n)$, we have

\begin{equation}\label{8}
\omega = \sum_{j=1}^k \frac{\sqrt{-1}dz_j \wedge d\bar{z}_j}{|z_j|^2\log^2|z_j|^2} + 
\left(\lambda\omega_0 - \sum_{j=k+1}^N \ddbar{\log\rho_j} \right) + O(\rho_1^{-1}+ \cdots + \rho_k^{-1})
\end{equation}

In particular, $\omega$ is quasi-isometric to $\omega_{mdl}$ on $U\backslash D$, i.e., there exists some positive constant $C=C(U,M,\omega)$ such that 
\[
C^{-1}\omega_{mdl} \leq \omega \leq C\omega_{mdl}.
\]
\end{lemma}

\begin{proof}
Note that $\lambda \omega_0 - \sum_{j=k+1}^N \ddbar{\log \rho_j}$ is smooth on $D$. For $1 \leq j \leq k$, $|\sigma_j|^2 = e^{f_j}|z_j|^2$ for some smooth $f_j$ through $D$. Thus, $\rho_j = -\log |z_j|^2 - f \sim -\log |z_j|^2$. A simple computation shows that
\[
\begin{aligned}
 -\ddbar{\log \rho_j} 
= &\ \frac{\sqrt{-1}dz_j\wedge d\bar{z}_j}{|z_j|^2\rho_j^2} -\frac{\ddbar{f}}{\rho_j} \\
&\ +  \frac{\sqrt{-1}(z_jdz_j\wedge \bar{\partial f} + \bar{z}_j \partial f\wedge d\bar{z}_j + |z_j|^2\partial f\wedge \bar{\partial}f)}{|z_j|^2\rho_j^2} \\
= &\ \frac{\sqrt{-1}dz_j\wedge d\bar{z}_j}{|z_j|^2\log^2|z_j|^2} + O(\rho_j^{-1}).
\end{aligned}
\]
Sum up for $j=1, \ldots, k$, we obtain (\ref{8}). The second part follows easily from (\ref{8}).
\end{proof}

In light of Lemma \ref{lemma2}, we readily see properties of such metric of Poincar\'{e} type: it is complete, has finite volume, and its injectivity radius goes to 0. In the next lemma, we collect some facts about the reference metric $\omega$ which we will use later.

\begin{lemma}
Let $(M,\omega)$ be constructed above. There exists some positive constant $B$ such that
\begin{itemize}
\item[(1)] $\inf_{i\neq j} R_{i\bar{i}j\bar{j}}\geq -B$;
\item[(2)] $|R|\leq B$;
\item[(3)] $\sup \frac{|\nabla \rho|}{\rho} \leq B.$
\end{itemize}
where $R_{i\bar{i}j\bar{j}}$ and $R$ are the holomorphic sectional curvature and scalar curvature of $(M,\omega)$, respectively. 
\end{lemma}

\begin{proof}
We only need to consider near the divisor $D$. Suppose $U$ is an open neighborhood of some point on $D$ such that $U\cap D = U\cap (D_1\cup\cdots\cup D_k)=\cup_{j=1}^k \{z_j=0\}$, where $z_j=0$ is the equation of $D_j$ in $U$. The metric $\omega$ has the asymptotics in $U\backslash D$ as stated in (\ref{8}). Note that $(\Delta^*)^k$ with standard cusp metric $\sum_{j=1}^k\frac{\sqrt{-1}dz_j\wedge d\bar{z}_j}{|z_j|^2\log^2|z_j|^2}$ has holomorphic sectional curvature $-1$. Hence, the holomorphic sectional curvature of $\omega$ on $U\backslash D$ is bounded from below and the scalar curvature bounded by some constant on $U\backslash D$.

To see (3), let us assume $U\cap D= U\cap D_1 = \{z_1=0\}$ and $\omega$ is the local model metric $\frac{\sqrt{-1}dz_1\wedge d\bar{z_1}}{|z_1|^2\log^2|z_1|^2} + \sum_{j=2}^n \sqrt{-1}dz_j\wedge d\bar{z}_j$ for simplicity. Set $\rho=\rho_1\rho'$ where $\rho'=\rho_2\cdots \rho_N$ is smooth on $D$. We have $\frac{\nabla \rho}{\rho}=\frac{\nabla \rho_1}{\rho_1}+\frac{\nabla \rho'}{\rho'}$. Note that $\rho'$ is bounded and $|\nabla \rho'|^2= (|z_1|^2\log^2|z_1|^2) |\partial_{z_1}\rho'|^2+\sum_{j=2}^n |\partial_{z_j} \rho'|^2 \leq \sum_{j=1}^n |\partial_{z_j} \rho'|^2$ is bounded on $U\backslash D$. On the other hand, $\rho_1 = -\log|z_1|^2 -f \sim -\log|z_1|^2$, we have $|\nabla \rho_1|^2 \leq -|z_1|^2\log^2|z_1|^2 (|z_1|^{-2}+ |\partial_{z_1} f|^2) + O(1) = \rho_1^2 + o(\rho_1)$. Hence, $\frac{|\nabla \rho_1|}{\rho_1} \rightarrow 1$ as $z_1\rightarrow 0$. Therefore, we have $\frac{|\nabla \rho|}{\rho}$ bounded by some constant on $U\backslash  D$.

Finally, we cover the divisor $D$ by finitely many such $U$'s and take the maximum of those constants as $B$, thus obtain (1), (2) and (3). 
\end{proof}

\subsection{Quasi-coordinates}
The viewpoint of quasi-coordinates has been adopted to study complete K\"{a}hler manifolds by Tian-Yau \cite{Tian-Yau} and R. Kobayashi\cite{Kobayashi}. They are useful to define likewise H\"{o}lder space, while the usual coordinate charts has inconvenience because of the injectivity radius going to zero. In these quasi-coordinates the interior Schauder estimate on complete manifold can be reduced to that on a bounded domain in Euclidean space. In this paper, we will use the quasi-coordinates to establish the weighted Sobolev inequality by reducing them to the bounded domains in Euclidean space.

\begin{definition}
Let $V$ be an open set in $\mathbb{C}^n$. A holomorphic map $\Phi$ from $V$ into a complex manifold $M$ of dimension $n$ is called a \emph{quasi-coordinate map} iff it is  of maximal rank everywhere in $V$. The pair $(V; \mbox{Euclidean coordinates of } \mathbb{C}^n)$ is called a local quasi-coordinate of $M$.
\end{definition}

To construct quasi-coordinates for $(M,\omega)$, we begin with the punctured disc $\Delta^*=\Delta-\{0\}$ with the standard cusp metric $\omega_{cusp}= \frac{\sqrt{-1}dz\wedge d\bar{z}}{|z|^2\log^2|z|^2}$. We start from the universal covering map $\pi:\Delta\rightarrow \Delta^*$, given by $\pi(w)=\exp\left(\frac{w+1}{w-1}\right)$. Formally, it sends 1 to 0. The idea is to restrict $\pi$ to the fixed disc $\frac{3}{4}\Delta$ (disc with radius $3/4$), and compose it with a biholomorphism $\psi_\delta$ of $\Delta$ sending $0$ to $\delta$, where $\delta\in (0,1)$ is a real parameter. To write the formula, we set $\psi_\delta(w)=\frac{w+\delta}{1+\delta w}$, so that the quasi-coordinate maps are given by
\begin{gather}\label{equ19}
\varphi_\delta =  \pi\circ\psi_\delta : \frac{3}{4}\Delta \rightarrow \Delta^*,\, \,  
\varphi_\delta(w) = \exp\left(-\frac{1+\delta}{1-\delta}\frac{w+1}{w-1}\right)
\end{gather}

It is easy to check the following properties of the quasi-coordinate maps.
\begin{itemize}
\item[(1)] ${\varphi_\delta}^*\omega_{cusp} = \frac{\sqrt{-1}dw\wedge d\bar{w}}{(1-|w|^2)^2}$
is independent on $\delta$ and $C^\infty$-quasi-isometric to the Euclidean metric on the ball.

\item[(2)] ${\varphi_\delta}^*(-\log(|z|^2))=2\frac{1+\delta}{1-\delta}\frac{1-|w|^2}{|1-w|^2}$ is mutually bounded with $\frac{1}{1-\delta}$ with fixed factor, i.e., there is constant $C>0$ independent of $\delta$ such that $\frac{1}{C(1-\delta)} \leq {\varphi_\delta}^*(-\log|z|^2) \leq \frac{C}{1-\delta}$ on $\frac{3}{4}\Delta$.

\item[(3)] There exists a constant $\kappa>0$ small (indeed $\kappa\leq e^{-25/7}$) such that $\kappa\Delta^*\subset \cup_{\delta\in (0,1)} \varphi_\delta (\frac{3}{4}\Delta)$. 
\end{itemize}

Now let us consider a crossing $D_1\cap \cdots \cap D_k$ of codimension $k$. For any point on such a crossing, we take an open neighborhoond $U$ such that 
$U\cap D_j=\{z_j=0\},\ j=1,\ldots,k$. Under the coordinates of $(z_1,\ldots,z_n)$, $U$ can be taken so that $U\backslash D$ is biholomorhic to $(\kappa \Delta^*)^k\times \Delta^{n-k}$. Let $\frac{3}{4}\mathcal{P}_k$ denote the polydisc $(\frac{3}{4}\Delta)^k\times \Delta^{n-k}$ in $\mathbb{C}^n$. Let $\delta=(\delta_1,\ldots,\delta_k)\in (0,1)^k$ be a multi-index. Then the quasi-coordinate map is given by
\begin{gather*}
\Phi_\delta: \polydisc \rightarrow (\kappa\Delta^*)^k\times \Delta^{n-k}\\
(w_1,\ldots,w_k,z_{k+1},\ldots,z_m)\mapsto \left(\varphi_{\delta_1}(w_1),\ldots,\varphi_{\delta_k}(w_k),z_{k+1},\ldots,z_m \right)
\end{gather*}
Note that $\cup_\delta \Phi_\delta(\polydisc)$ covers $U\backslash D$.

The quasi-coordinate ``atlas" of $(M,\omega)$ is obtained by covering an open neighborhood of $D$ by our local quasi-coordinate charts, and covering the complement by a finite number of unit balls in $\mathbb{C}^n$.

At this stage we introduce the H\"{o}lder norms and H\"{o}lder spaces using the previously introduced quasi-coordinates for later use.
\begin{definition}
For a non-negative integer $k$, and a real number $\alpha \in (0,1)$, we define
\[
||u||_{C^{k,\alpha}_{qc}(M)} := \sup \{ ||u\circ \Phi||_{C^{k,\alpha}(V)}: (V,\Phi) \text{ \emph{is a quasi-coordinate map}} \}
\]
where the supremum is taken over all our quasi-coordinate maps $(V,\Phi)$. The $||\dot||_{C^{k,\alpha}(V)}$ is the usual H\"{o}lder norm on $V\subset \mathbb{C}^n$. The H\"{o}lder space $C^{k,\alpha}_{qc}(M)$ is 
\[
C^{k,\alpha}_{qc}(M) := \{ u\in C^k_{loc}(M): ||u||_{C^{k,\alpha}_{qc}(M)}< \infty\}.
\]
\end{definition}
 
\subsection{Weighted Sobolev inequality}

The weighted Sobolev inequality stated in this subsection is first proved by Auvray in \cite{AUV}, Lemma 4.4. The following lemma, which is crucial for Auvray's proof, however, is just stated without a proof. For  readers' convenience and completeness, we give a proof here.

We briefly set up the setting in this section. Consider a point on a crossing of codimension $k$. Let $U$ be a polydisc centered at this point so that $U\backslash D$ is biholomorphic to $(\kappa\Delta^*)^k\times \Delta^{n-k}$. For multi-indices $\delta\in (0,1)^k$, let $\Phi_\delta: \frac{3}{4}\mathcal{P}_k \rightarrow U\backslash D$ be the quasi-coordinate maps defined in subsection 2.2. Let $\frac{1}{2}\mathcal{P}_k$ denote the polydisc $(\frac{1}{2}\Delta)^k\times \Delta^{n-k}$. 

\begin{lemma}\label{lem3}
There exists a constant $c>0$ depending on $(U,n,\omega)$ and a sequence of multi-indices $(\delta_\ell)$, $\delta_\ell=((\delta_1)_{\ell_1},\ldots,(\delta_k)_{\ell_k})$ such that for any $f \in L^1(U\backslash D)$,
\begin{multline}
c^{-1}\sum_{\ell} A_{\delta_\ell}\int_{\frac{3}{4}\mathcal{P}_k} |{\Phi_{\delta_\ell}}^* f| dV_{g_0}
\leq \int_{U\backslash D} |f| dV 
\leq c \sum_{\ell}  A_{\delta_\ell} \int_{\frac{1}{2}\mathcal{P}_k} |{\Phi_{\delta_\ell}}^* f| d{V}_{g_0}
\end{multline}
where $A_{\delta_\ell} = \prod_{j=1}^k \left(1-(\delta_j)_{\ell_j}\right)$, and $g_0$ is the Euclidean metric on $\mathbb{C}^n$. Note that ${\Phi_{\delta_\ell}}^*\rho$ is mutually bounded with $A_{\delta_\ell}^{-1}$, and the pull-back metric ${\Phi_\delta}^*g$ is $C^\infty$-quasi-isometric to the Euclidean metric $g_0$ with fixed factor independent of $\delta$. In particular, if $(f\rho) \in L^1(U\backslash D)$, then
\begin{equation}
c^{-1} \sum_{\ell} \int_{\polydisc} |{\Phi_{\delta_\ell}}^* f| dV_{g_0}
\leq \int_{U\backslash D} |f|\rho dV 
\leq c \sum_{\ell} \int_{\frac{1}{2}\mathcal{P}_k} |{\Phi_{\delta_\ell}}^* f| d{V}_{g_0}
\end{equation}
\end{lemma}

\begin{proof}
The proof is quite technical. We shall begin with the simplest case to show the key point of the proof. Let us consider the model case of Poincar\'{e} disc with standard Poincar\'{e} metric. In this case, $\delta\in (0,1)$ is a real parameter. Let $\sigma =\frac{1+\delta}{1-\delta}$. The quasi-coordinate map $\Phi_{\delta}$ is just $\varphi_\delta(w) = \exp\left(-\frac{1+\delta}{1-\delta}\frac{w+1}{w-1}\right)=\exp \circ \zeta_\sigma(w)$, where $\zeta_\sigma(w) = -\sigma \frac{w+1}{w-1}$. The covering map $\exp:\mathbb{C}\rightarrow \Delta^*$ maps any horizontal strip of width $2\pi$ onto the punctured disc. 
Under the map $\zeta_\sigma$, the disc $\frac{1}{2}\Delta$ is mapped biholomorphically to another disc $B_\sigma$ with center $(-\frac{5}{3}\sigma,0)$ and radius $\frac{4}{3}\sigma$. The union of such balls covers the half strip 
$$(-\infty, \log \kappa)\times (-\pi,\pi]\subseteq \mathbb{R}^2.$$

The idea is to pull back and take the integral over the strip, then cut the strip into small pieces so that each piece is contained in some ball. Each piece is covered by the ball a finite number of multiplicity proportional to the radius of the ball, hence proportional to $\sigma$. To make it precise, let $\hat{g}$ be the pullback metric on $\mathbb{C}$ of the standard cusp metric on the punctured disc by the covering map $\exp$. Let $\hat{f} = f\circ\exp$. Then
\begin{equation*}
\int_{\kappa\Delta^*} |f| d{V} = \int_{(-\infty,\log \kappa)\times (-\pi,\pi]} |\hat{f}|d{V}_{\hat{g}}
\end{equation*}
The rectangle 
$$I_\sigma \times J_\sigma := \left(-\frac{20}{9}\sigma,-\frac{10}{9}\sigma \right] \times \left(-\frac{\sqrt{119}}{9}\sigma,\frac{\sqrt{119}}{9}\sigma\right)$$
is contained in the ball. We have
\begin{equation*}
\int_{I_\sigma\times (-\pi,\pi]}|\hat{f}|dV_{\hat{g}} \leq \frac{18\pi}{\sqrt{119}\sigma}\int_{I_\sigma\times J_\sigma} |\hat{f}|d{V}_{\hat{g}} \leq c\sigma^{-1}\int_{\frac{1}{2}\Delta} |{\varphi_\delta}^*f|dV_{g_0}
\end{equation*}
Now pick a sequence $(\sigma_\ell)$ such that $\sigma_1 = -\frac{3}{5}\log\kappa,\, \sigma_{\ell+1}=2\sigma_\ell$. Then
$(-\infty, \log \kappa)\subseteq \cup I_{\sigma_\ell}$. Hence,
$$\int_{(-\infty,\log \kappa)\times (-\pi,\pi]} |\hat{f}|dV_{\hat{g}}\leq \sum_{\ell=1}^\infty \int_{I_{\sigma_\ell \times (-\pi,\pi]}}|\hat{f}|dV_{\hat{g}}.$$
Namely,
\begin{equation*}
\int_{\kappa \Delta^*} |f|dV \leq c\sum_\ell \sigma_\ell^{-1} \int_{\frac{1}{2}\Delta}|{\varphi_{\delta_\ell}}^*f|dV_{g_0}.
\end{equation*}
Note that $\sigma^{-1}$ is equivalent to $(1-\delta)$. Hence,
\begin{equation}\label{equ20}
\int_{\kappa \Delta^*} |f| dV \leq c'\sum_\ell \int_{\frac{1}{2}\Delta} (1-\delta_\ell)|{\varphi_{\delta_\ell}}^*f|dV_{g_0}.
\end{equation}
This gives us the right half inequality.

On the other hand, $\frac{3}{4}\Delta$ is mapped by $\phi_{\sigma}$ to a ball with center $(-\frac{25}{7}\sigma,0)$ and radius $\frac{24}{7}\sigma$. Let 
$$I'_\sigma = \left(-7\sigma,-\frac{1}{7}\sigma\right],\quad J'_\sigma = \left(-\frac{24}{7}\sigma,\frac{24}{7}\sigma\right).$$ 
Then
\begin{equation}\label{equ21}
\int_{\frac{3}{4}\Delta} |{\phi_\sigma}^*f| dV_{g_0} \leq \int_{I'_\sigma\times J'_\sigma} |\hat{f}|dV_{\hat{g}}\leq \frac{48\sigma}{7\pi}\int_{I'_\sigma\times (-\pi,\pi]} |\hat{f}|dV_{\hat{g}}
\end{equation}
For the same sequence $(\sigma_\ell)$ in (\ref{equ20}), these $I'_{\sigma_\ell}$ are overlapped. But each of them only intesect with finitly many others. To see this,
two balls do not intersect if the distance of the centers is bigger than the sum of their radius. The balls $B_{\sigma_{\ell_1}}$ and $B_{\sigma_{\ell_2}}$ do not intersect (suppose $\ell_1>\ell_2$) if
\begin{gather*}
\frac{25}{7}(2^{\ell_1}-2^{\ell_2})> \frac{24}{7}(2^{\ell_1}+2^{\ell_2})\quad
\Rightarrow \quad \ell_1-\ell_2>\log_249
\end{gather*}
Hence, each ball intersects with no more than $2\log_2 49 <16$ balls. Therefore,
\begin{equation*}
\sum_\ell \int_{I'_{\sigma_\ell}\times (-\pi,\pi]} |\hat{f}|dV_{\hat{g}} \leq 16\int_{(-\infty,\log\kappa)\times (-\pi,\pi]} |\hat{f}|dV_{\hat{g}} =16 \int_{\kappa\Delta^*} |f|dV
\end{equation*}
Combine with (\ref{equ21}), and that $\sigma^{-1}$ is equivalent to $(1-\delta)$, we have
\begin{equation}
\sum_\ell (1-\delta_\ell) \int_{\frac{3}{4}\Delta}|{\varphi_{\sigma_\ell}}^* f|dV_{g_0} \leq c'' \int_{\kappa \Delta^*} |f|dV 
\end{equation}
We get the left half inequality.

Now come back to our setting. When $k=1$, then $U\backslash D = \kappa \Delta^*\times \Delta^{n-1}$. We can assume the metric $g$ on $U\backslash D$ is the standard model metric $\frac{\sqrt{-1}dz_1\wedge d\bar{z}_1}{|z_1|^2\log^2|z_1|^2}+\sum_{j=2}^n \sqrt{-1}dz_j\wedge d\bar{z}_j$. The argument goes exactly the same after multiplying each integral region by $\Delta^{n-1}$; if $k\geq 2$, we write the integral as multiple integral and treat the variables in this manner in order. 
\end{proof}

Now we prove the following weighted Sobolev inequality on $(M,\omega)$.
\begin{lemma}[cf. \cite{AUV}, Lemma 4.4]\label{sob_inequ}
There exists a positive constant $C= C(p,n,\omega)$ such that for any function $v\in W^{1,p}_{loc}$, 
\begin{equation}
\left(\int_M |v|^{q} \rho dV\right)^{1/q} \leq C \left(\int_M \left(|v|^{p}+|\nabla v|^{p}\right) \rho dV\right)^{1/p}
\end{equation}
as long as $q\geq p$ and $\frac{1}{p}\leq \frac{1}{2n}+\frac{1}{q}$.
\end{lemma}

\begin{proof}
Away from the divisor, $\rho$ is bounded, then it is just the usual Sobolev inequality in open domains of $\mathbb{C}^n$. Hence, we only need to look at what happens near the divisor. Suppose a point on a crossing of codimension $k$. Consider a small polydisc $U$ around this point and cover $U\backslash D$ by a union
$\cup_{\delta} \Phi_{\delta}(\frac{3}{4}\mathcal{P}_k)$
where $\delta=(\delta_1,\ldots,\delta_k)$. We can assume that the metric on $U\backslash D$ is just the standard model metric. Then all the pullback ${\Phi_\delta}^* g$ give the same metric on $\frac{3}{4}\mathcal{P}_k$, , which is quasi-equivalent to the Euclidean metric on $\mathbb{C}^n$. On $\frac{3}{4}\mathcal{P}_k$ we have the standard Sobolev inequality
\begin{equation}
||f||_{L^q(\frac{3}{4}\mathcal{P}_k)} \leq C(n,p) ||f||_{W^{1,p}(\frac{3}{4}\mathcal{P}_k)}
\end{equation}
for any $q$ with $\frac{1}{p}\leq \frac{1}{q}+\frac{1}{2n}$. Then by Lemma \ref{lem3}, we can take a sequence $(\delta_\ell)$ and a positive constant $c$ depending on $(U,n,\omega)$, so that
\begin{gather}
\begin{aligned}
\int_{U\backslash D} |v|^q\rho dV 
\leq&\ c \sum_{\ell} ||{\Phi_{\delta_\ell}}^* v||^q_{L^q(\frac{1}{2}\mathcal{P}_k)}\\
\leq&\ cC(n,p)^q\sum_{\ell} ||{\Phi_{\delta_\ell}}^* v||^q_{W^{1,p}(\frac{3}{4}\mathcal{P}_k)}\\
\leq&\ cC(n,p)^q \left(\sum_{\ell} ||{\Phi_{\delta_\ell}}^* v||_{W^{1,p}(\frac{3}{4}\mathcal{P}_k)}^p\right)^{q/p} \text{ since $q\geq p$}\\
\leq&\ c^2C(n,p)^q \left(\int_{U\backslash D} (|v|^p+|\nabla v|^p) \rho dV\right)^{q/p}.
\end{aligned}
\end{gather}
Therefore, 
\begin{equation}
\begin{aligned}
\left(\int_{U\backslash D} |v|^q \rho dV\right)^{1/q} \leq&\ c^{2/q}C(n,p) \left(\int_{U\backslash D} (|v|^p+|\nabla v|^p) \rho dV\right)^{1/p}\\
\leq&\ c^{1/p}C(n,p) \left(\int_{U\backslash D} (|v|^p+|\nabla v|^p) \rho dV\right)^{1/p}.
\end{aligned}
\end{equation}

We can cover $D$ by finitely many such $U$'s and cover the complement of the union of these $U$'s by finitely many unit balls in $\mathbb{C}^n$. Take the constant $C$ to be the maximum of Sobolev constant in each $U$ and each unit ball. Then, the constant can be made only depend on $(p,n,\omega)$. 
\end{proof}

As a corollary of Lemma \ref{sob_inequ}, we show that function $F$ with $I(F,p_0)<\infty$ for some $p_0>2n$ are bounded.

\begin{corollary}\label{cor2.3.1}
Suppose $F$ satisfies $I(F,p_0)<\infty$ for some $p_0> 2n$. Then $||F||_{L^\infty}\leq C(I(F,p_0),p_0,n,\omega)$.
\end{corollary}
\begin{proof}
Note that $\int (|F|^{p_0}+|\nabla F|^{p_0})\rho dV \leq I(F,p_0) <\infty$ since $p_0> 2n$ and $\rho\geq 1$. Moreover, we have $1/p_0 \leq 1/q + 1/2n$ for any $q\geq p_0$. The lemma follows by taking $v=F$ in Lemma \ref{sob_inequ} then letting $q\rightarrow \infty$.
\end{proof}

\section{Uniform $C^0$ estimate}
In this section we prove the uniform $C^0$ estimate for $\phi_\epsilon$. 

\begin{proposition}\label{prop}
Suppose $F\in C^\infty_c(M)$ satisfies $I(F,p_0)<\infty$. Let $\phi_\epsilon$ be the solution for equation (\ref{equ1}). Then there exists a constant $C = C(I(F,p_0),p_0,n,\omega)$ such that for all $\epsilon\in (0,1]$, 
$$||\phi_\epsilon||_{L^\infty}\leq C.$$
\end{proposition}

\remark This result is proved in \cite{AUV} where the constant $C$ depends on $||F||_{C^0_\nu}$, $\nu$, $n$, and $\omega$. Here $\nu>0$ and $||F||_{C^0_\nu}:= \sup_M \rho^\nu |F|$. In our case the data of $F$ is to some extent weaker. It only depends on the integral bound $I(F,p_0)$. The proof is just a slight modification in dealing with a term containing $F$ in a H\"{o}lder inequality. 

\begin{proof}
For simiplicity of notations, we will drop the subscript $\epsilon$ from $\phi_\epsilon$. The integrals without subscript is taking on $M$ for granted.
First of all, Proposition 4.1 in \cite{AUV} holds, namely, 
$||\phi||_{L^2}\leq C=C(||F||_{L^\infty},p_0,n,\omega))$. Note that by Corollary \ref{cor2.3.1}, $||F||_{L^\infty} \leq C(I(F,p_0),p_0,n,\omega)$.
By Proposition 4.2 of \cite{AUV},
\begin{equation}\label{equ26}
\int \big|\nabla |\phi|^{p/2}\big|^2 dV \leq \frac{p^2}{4(p-1)}\int |\phi|^{p-2}\phi(1-e^F)dV.
\end{equation}
An easy computation yields
\begin{equation}\label{equ27}
\int \big|\nabla (|\phi|^{p/2}\rho^{-1/2})\big|^2 \rho dV \leq 2\int \big|\nabla |\phi|^{p/2} \big|^2 dV + 2\int |\phi|^p \left(\frac{|\nabla \rho|}{\rho}\right)^2dV.
\end{equation}
Applying wighted Sobolev inequality to $|\phi|^{p/2}\rho^{-1/2}$, we obtain
\begin{equation}\label{equ28}
\int |\phi|^{{2pn}/{(2n-1)}} \rho^{-{1}/{(2n-1)}} dV \leq C \left( \int \big|\nabla (|\phi|^{p/2}\rho^{-1/2})\big|^2 \rho dV + \int |\phi|^p dV \right).
\end{equation}
Applying a version of Poincar\'{e} inequality developed in \cite{AUV}, Lemma 1.10 (with a mean term) to $|\phi|^{p/2}$, we have
\begin{equation}\label{equ29}
\int |\phi|^p dV \leq C\int \big|\nabla |\phi|^{p/2}\big|^2 dV + Vol(M)^{-1} \left( \int |\phi|^{p/2} dV \right)^2
\end{equation}
Note that $\sup \frac{|\nabla \rho|}{\rho} \leq C$. Collect (\ref{equ26}), (\ref{equ27}), (\ref{equ28}) and (\ref{equ29}), we get
\begin{equation}
\begin{aligned}
\int |\phi|^{{2pn}/{(2n-1)}} \rho^{-{1}/{(2n-1)}} dV 
\leq&\ C\int |\nabla |\phi|^{p/2}|^2 dV + C' \left(\int |\phi|^{p/2}dV\right)^2\\
\leq&\ Cp\int |\phi|^{p-1}|e^F-1|dV + C' \left(\int |\phi|^{p/2}dV\right)^2
\end{aligned}
\end{equation}
Let $d\mu$ denote the measure $\rho^{-1/(2n-1)}dV$, we can rewrite the above as
\begin{equation}
\int |\phi|^{{2pn}/{(2n-1)}} d\mu \leq C\int |\phi|^{p-1}|e^F-1|\rho^{1/(2n-1)}d\mu + C'\left(\int |\phi|^{p/2}\rho^{1/(2n-1)}d\mu \right)^2
\end{equation}
We have $|e^F-1| \leq C|F|$, where $C$ depends on $||F||_{L^\infty}$. Let $q_0>0$ such that $1/p_0+1/q_0 =1$. By H\"{o}lder inequality,
\begin{equation}
\begin{aligned}
\int |\phi|^{p-1}|F|\rho^{1/(2n-1)}d\mu 
\leq&\ \left( \int |F|^{p_0}\rho^{p_0/(2n-1)}d\mu \right)^{1/p_0} \left(\int |\phi|^{(p-1)q_0}d\mu \right)^{1/q_0} \\
\leq&\ (I(F,p_0))^{1/p_0} \left(\int |\phi|^{(p-1)q_0}d\mu \right)^{1/q_0} \\
\leq&\ C ||\phi||_{L^{pq_0}(d\mu)}^p.
\end{aligned}
\end{equation}
Let $1/{p_1}+{1}/{(2q_1)}=1$ and ${n}/({2n-1})<q_1< {2n}/({2n-1})$, then $p_1<2n $. By H\"{o}lder inequality,
\begin{equation}
\begin{aligned}
\left(\int |\phi|^{p/2} \rho^{1/(2n-1)} d\mu \right)^2 
\leq&\  \left( \int \rho^{p_1/(2n-1)}d\mu \right)^{1/p_1} \left( \int |\phi|^{pq_1} d\mu \right)^{1/q_1}\\
\leq&\  C||\phi||_{L^{pq_1}(d\mu)}^p.
\end{aligned}
\end{equation}
Here $\int \rho^{p_1/(2n-1)}d\mu = \int \rho^{(p_1-1)/(2n-1)}dV < \infty$ since $p_1<2n$.
Hence,
\begin{equation}
||\phi||_{L^{p\frac{2n}{2n-1}}(d\mu)}^p \leq Cp ||\phi||^p_{L^{pq_0}(d\mu)} + C'||\phi||_{L^{pq_1}(d\mu)}^p
\end{equation}
with $q_0, q_1 < 2n/(2n-1)$.
Take $q_2= \max(q_0,q_1)$. Then $q_2< 2n/(2n-1)$ and
\begin{equation}
||\phi||_{L^{p\frac{2n}{2n-1}}(d\mu)} \leq C^{1/p}p^{1/p}||\phi||_{L^{pp_2}(d\mu)}.
\end{equation}
Hence, by standard iteration process we have
\[
||\phi||_{L^\infty}\leq C||\phi||_{L^2(d\mu)}\leq C||\phi||_{L^2(dV)}\leq C.
\]
\end{proof}

\section{Uniform $C^1$ estimate}
In this section we prove Theorem \ref{thm2}. For simiplicity of notations, we will drop the subscript $\epsilon$ from $\phi_\epsilon$. The constant $C$ may vary from line to line, but only depends on $I(F,p_0),\ p_0,\ \omega$ and $n$. 

We follow a computation in \cite{Chen-He}, section 3. We have the following inequality (see \cite{Chen-He}, equation (3.11))
\begin{equation}
\begin{aligned}
\Delta'\left(e^{-A(\phi)}|\nabla \phi|^2 \right)\geq &\ e^{-A(\phi)}|\nabla \phi|^2\left(-A''|\nabla'\phi|_\phi^2+(A'-B)tr_{g'}g\right)\\
&+(2A'-B)e^{-A(\phi)}|\nabla' \phi|_\phi^2  -\left((n+2)A'+2\epsilon\right)e^{-A(\phi)}|\nabla \phi|^2\\
&+ e^{-A(\phi)}(\Delta \phi -n + tr_{g'}g) - 2e^{A(\phi)}|\nabla F||\nabla \phi|.
\end{aligned}
\end{equation}
where $B>0$ is constant so that $\inf_{i\neq j} R_{i\bar{i}j\bar{j}} \geq -B$. Let $C_0$ be a fixed positive constant such that $C_0=1+ ||\phi||_{L^\infty}$. Choose
$$A(t)=(B+2)t -\frac{t^2}{2C_0}.$$
It follows that
$$
B+1 \leq A'(\phi)=B+2-\frac{\phi}{C_0}\leq B+3,\ A''(\phi)=-\frac{1}{C_0}.$$
It is easy to see that 
\begin{equation}\label{equ3}
tr_{g'}g \geq \exp(-F/(n-1))(tr_gg')^{1/(n-1)}.
\end{equation}
By (\ref{equ3}), we compute
\begin{equation}\label{equ15}
\begin{aligned}
&\ -A''|\nabla'\phi|_\phi^2 + (A'-B)tr_{g'}g\\
\geq &\ \frac{1}{C_0}|\nabla'\phi|_\phi^2 + \exp(-F/(n-1))(tr_gg')^{1/(n-1)}\\
\geq &\ n \left( \frac{\exp(-F)}{C_0(n-1)^{n-1}}|\nabla'\phi|_\phi^2 (tr_gg')\right)^\frac{1}{n}\\
\geq &\ C_1|\nabla \phi|^{2/n}
\end{aligned}
\end{equation}
for some $C_1$ depending on $||F||_{L^\infty}, ||\phi||_{L^\infty}$ and $n$.
Note that $n+\Delta \phi = tr_gg'$. By dropping the nonnegative terms $(2A'-B)e^{-A(\phi)}|\nabla' \phi|_\phi^2$ and $e^{-A(\phi)}tr_{g'}g$, and taking (\ref{equ15}) into account, we have
\begin{equation}
\begin{aligned}
\Delta'\left(e^{-A(\phi)}|\nabla \phi|^2 \right)\geq &\ C_1 e^{-A(\phi)}|\nabla \phi|^{2+\frac{2}{n}}
-\{(n+2)A'+2\}e^{-A(\phi)}|\nabla \phi|^2\\
&\ + e^{-A(\phi)}(tr_gg' -2n) - 2e^{-A(\phi)}|\nabla F||\nabla \phi|.
\end{aligned}
\end{equation}
We can intepolate $|\nabla \phi|^2$ by $|\nabla \phi|^{2+2/n}$ and constants, i.e.,
\begin{equation}
|\nabla \phi|^2 \leq \varepsilon |\nabla \phi|^{2+\frac{2}{n}} + C(\varepsilon).
\end{equation}
Let $u= \exp(-A(\phi))|\nabla \phi|^2$, then
\begin{equation}\label{equ4}
\Delta' u \geq (C_1u^{1+1/n}-C_2) + C_3tr_gg' -C_4 |\nabla F|u^{1/2}
\end{equation}
Now multiplying (\ref{equ4}) by $u^p,\ p>0$ and integration by parts,
\begin{equation}
\begin{aligned}
&\ -\int_M pu^{p-1}|\nabla' u|^2_\phi dV' + \int_M \nabla'(u^p\nabla' u) dV'\\
=&\ \int_M u^p\left(C_1 u^{1+\frac{1}{n}}-C_2\right) dV' + C_3 \int_M u^p(tr_gg')dV' - C_4\int_M |\nabla F|u^{p+\frac{1}{2}} dV'
\end{aligned}
\end{equation}
By Gaffney-Stokes, $\int_M \nabla'(u^p\nabla' u) dV'=0$. Hence,
\begin{equation}
\begin{aligned}
&\int_M pu^{p-1}|\nabla' u|^2_\phi + C_3u^p(tr_gg') dV' \\
\leq &\ C_4\int_M |\nabla F|u^{p+\frac{1}{2}} dV'-\int_M u^p(C_1 u^{1+\frac{1}{n}}-C_2)dV'
\end{aligned} 
\end{equation}
Note that pointwisely we have
\begin{equation}
|\nabla' u|^2_\phi + (tr_gg') \geq 2|\nabla u|.
\end{equation}
Hence,
\begin{equation}
\int_M 2\sqrt{C_3p}u^{p-\frac{1}{2}}|\nabla u|dV' \leq C_4\int_M |\nabla F|u^{p+\frac{1}{2}} dV' -\int_M u^p(C_1u^{1+\frac{1}{n}}-C_2)dV_\phi.
\end{equation}
Note that $dV$ and $dV'$ are equivalent, hence
\begin{equation}
\int_M |\nabla u^{p+\frac{1}{2}}| dV \leq C_5\sqrt{p}\int_M |\nabla F|u^{p+\frac{1}{2}}dV -  C_6\sqrt{p}\int_M u^p(u^{1+\frac{1}{n}}-C_7)dV
\end{equation}
Let us rewrite it as follows:
\begin{equation}\label{equ6}
\int_M |\nabla u^{p}| dV \leq C_5\sqrt{p}\int_M |\nabla F|u^{p}dV - C_6\sqrt{p}\int_M u^{p-\frac{1}{2}}(u^{1+\frac{1}{n}}-C_7)dV
\end{equation}
To get the $L^\infty$ bound, we use the iteration method. First, we have
\begin{equation}\label{equ5}
\int_M |\nabla (u^p\rho^{-1})| \rho dV =\int_M |\nabla u^p|dV + \int_M u^p\left(\frac{|\nabla \rho|}{\rho}\right) dV
\end{equation}
Apply the weighted Sobolev inequality (\ref{sob_inequ}) to function $u^p\rho^{-1}$, we have
\begin{equation}
\left(\int_M (u^p\rho^{-1})^{\frac{2n}{2n-1}}\rho dV\right)^{\frac{2n-1}{2n}} \leq \int_M |\nabla (u^p\rho^{-1})|\rho dV + \int_M (u^p\rho^{-1})\rho dV
\end{equation}
Taking (\ref{equ5}) into account, we have
\begin{equation}
\begin{aligned}
\left(\int_M u^{p\frac{2n}{2n-1}} \rho^{-\frac{1}{2n-1}}dV\right)^{\frac{2n-1}{2n}} \leq&\ \int_M |\nabla u^p|dV +\int_M u^p\left(1+\frac{|\nabla \rho|}{\rho}\right) dV\\
\leq&\ \int_M |\nabla u^p|dV +C_8\int_M u^p dV
\end{aligned}
\end{equation}
Taking (\ref{equ6}) into account, we have
\begin{equation}\label{equ17}
\begin{aligned}
&\left(\int_M u^{p\frac{2n}{2n-1}} \rho^{-\frac{1}{2n-1}}dV\right)^{\frac{2n-1}{2n}} \\
\leq &\ C_5\sqrt{p}\int_M u^{p}|\nabla F| dV 
- C_6\sqrt{p}\int_M u^{p-\frac{1}{2}}\left(u^{1+\frac{1}{n}} - C_8\sqrt{u}-C_7\right)dV \\
\leq &\  C_5\sqrt{p}\int_M u^{p}|\nabla F| dV 
- C_9\sqrt{p}\int_M u^{p-\frac{1}{2}}\left(u^{1+\frac{1}{n}} -C_{10}\right)dV
\end{aligned}
\end{equation}
There exists some constant $K$, such that when $u>K$, $u^{1+\frac{1}{n}} -C_{10}\geq 0$. Hence,
\begin{equation}\label{equ16}
\begin{aligned}
-\int_M u^{p-\frac{1}{2}}\left(u^{1+\frac{1}{n}} -C_{10}\right)dV
\leq&\ \int_{\{u\leq K\}} u^{p-\frac{1}{2}}\left(u^{1+\frac{1}{n}} -C_{10}\right)dV\\
\leq&\ C_{10}Vol(M)K^{p-1/2}
\end{aligned}
\end{equation}
Put (\ref{equ16}) into (\ref{equ17}) and let $d\mu = \rho^{-\frac{1}{2n-1}}dV$, then we have
\begin{equation}
\left(\int_M u^{p\frac{2n}{2n-1}}d\mu \right)^{\frac{2n-1}{2n}}
\leq C_5\sqrt{p}\left( \int_M u^p |\nabla F| \rho^{\frac{1}{2n-1}}d\mu  + C_{11}K^p\right)
\end{equation}
Now let $u=Kv$, we get inequality of $v$ that
\begin{equation}
\left(\int_M v^{p\frac{2n}{2n-1}}d\mu \right)^{\frac{2n-1}{2n}}
\leq C_{12}\sqrt{p}\left( \int_M v^p |\nabla F| \rho^{\frac{1}{2n-1}}d\mu  + 1\right)
\end{equation}
By H\"{o}lder inequality, we have
\begin{equation}
\int_M v^p|\nabla F|\rho^{\frac{1}{2n-1}}d\mu \leq \left(\int_M v^{pp_0} d\mu \right)^{1/p_0} 
\left(\int_M |\nabla F|^{p_0}\rho^{\frac{p_0}{2n-1}} d\mu\right)^{1/p_0}
\end{equation}
Note that 
$$\int_M |\nabla F|^{p_0} \rho^{\frac{p_0}{2n-1}}d\mu =\int_M |\nabla F|^{p_0}\rho^{\frac{p_0-1}{2n-1}}dV \leq C< \infty $$ 
and for $p_0 > 2n$, we have $p_0 < \frac{2n-1}{2n}$. Hence, 
\begin{equation}
||v||_{L^{p\frac{2n}{2n-1}}_{d\mu}}^p  \leq C_{12}\sqrt{p} \left( C||v||_{L^{pp_0}_{d\mu}}^p+1\right)
\end{equation}
Let $\beta = \frac{2n}{2n-1}p_0^{-1}$, then $\beta > 1$. Take a sequence $(p_k)$ with 
$p_k=p_0^{-1}\beta^{k-1}$ for $k\geq 1$. Then
\begin{equation}
||v||_{L^{p_0p_{k+1}}_{d\mu}}^{p_k}  \leq C_{12}\sqrt{p_k} \left( C||v||_{L^{p_0p_k}_{d\mu}}^{p_k}+1\right)
\end{equation}
Let $A_k = \max\{||v||_{L^{p_0p_k}_{d\mu}}, 1\}$, then
\begin{equation}
A_{k+1}  \leq C'^{1/p_k} p_k^{2/p_k}A_k.
\end{equation}
Hence, by iteration, we have
\begin{equation}
A_k \leq \left(\prod_{j=1}^{k-1} C'^{1/p_j}p_j^{2/p_j}\right) A_1.
\end{equation}
It is easy to check that 
$$\sum_{j=1}^\infty \frac{1}{p_j}\left(2\log p_j+\log C'\right)
=\sum_{j=0}^\infty \frac{p_0}{\beta^{j}}\left(2j\log \beta-2\log p_0 + \log C'\right)< \infty.$$
Therefore, we have $||v||_{L^{p_0p_k}_{d\mu}}\leq C$ for some positive constant $C$.
Let $k\rightarrow \infty$, we have
\begin{equation}
||v||_{L^\infty} \leq C||v||_{L^1_{d\mu}}.
\end{equation}
Note that $u=Kv$, hence
\begin{equation}
||u||_{L^\infty} \leq C||u||_{L^1_{d\mu}}.
\end{equation}
Note that $d\mu\leq dV$, and that $\phi$ has uniform $L^\infty$ bound, we have
\begin{equation}
\begin{aligned}
||u||_{L^1_{d\mu}}=&\ \int_M \exp(-A(\phi))|\nabla \phi|^2d\mu
\leq C\int_M |\nabla \phi|^2dV \\
\leq&\ -C\int_M \phi\Delta\phi dV 
= -C \int_M \phi(\Delta \phi+n-n) dV \\
\leq&\ C||\phi||_{L^\infty}\int_M (\Delta\phi+n) dV + Cn\int_M \phi dV \leq C
\end{aligned}
\end{equation}
Hence, we obtain $||u||_{L^\infty} \leq C_1$, namely, $||\nabla v||_{L^\infty}\leq C$, where is constant $C$ depends on $||\phi||_{L^\infty},\ ||F||_{L^\infty}, I(F,p_0),\ p_0,\ g,\ n$. By Corollary \ref{cor2.3.1} and Proposition \ref{prop}, the constant $C$ only depends on $I(F,p_0),\ p_0,\ g,\ n$.

\section{Uniform $C^2$ estimate}
In this section we give the uniform Laplacian estimate of $\phi_\epsilon$, namely, we prove Theorem \ref{thm3}. For the simplicity of notations, we ommit the subscript $\epsilon$ from $\phi_\epsilon$. The constant $C$ may differ from line to line.

Let $w= e^{-A\phi}(tr_g g')$ with some constant $A\geq -\inf_{i\neq\ell} R_{i\bar{i}\ell\bar{\ell}}+1$. We first show the following inequality holds.

\begin{lemma}\label{lem2}
There exist positive constants $\theta$ and $C$ depending only on $||\phi||_{L^\infty}$  and $||F||_{L^\infty}$, such that 
\begin{equation}
\Delta' w \geq \theta w^{\frac{n}{n-1}}-C-e^{-A\phi}\Delta F
\end{equation}
\end{lemma}

\begin{proof}
We start with Yau's inequality
\begin{equation}
\begin{aligned}
\Delta' (tr_g g') \geq&\ (\inf_{i\neq \ell} R_{i\bar{i}\ell\bar{\ell}})(tr_g g')(tr_{g'} g) + g^{j\bar{k}}Ric'_{j\bar{k}}  + \frac{1}{tr_g g'}|\nabla' tr_g g'|_\phi^2 \\
=&\ (\inf_{i\neq \ell} R_{i\bar{i}\ell\bar{\ell}})(tr_g g')(tr_{g'} g) + (R-\Delta(F+\epsilon \phi))  + \frac{1}{tr_g g'}|\nabla' tr_g g'|_\phi^2
\end{aligned}
\end{equation}
Hence,
\begin{equation}
\Delta' \log(tr_g g') \geq (\inf_{i\neq \ell} R_{i\bar{i}\ell\bar{\ell}})(tr_{g'}g)+ \frac{R-\Delta F-\epsilon\Delta\phi}{tr_g g'}
\end{equation}
Note that $\Delta' \phi = 2n - tr_{g'}g$ and $\Delta \phi = tr_gg' -2n$,
\begin{equation}
\begin{aligned}
\Delta' \left(\log(tr_g g') - A\phi\right) \geq&\ (\inf_{i\neq \ell} R_{i\bar{i}\ell\bar{\ell}}+ A)(tr_{g'}g)-(2nA+\epsilon) + \frac{R-\Delta f+2\epsilon n }{tr_g g'}\\
\geq &\ tr_{g'}g - (2nA+\epsilon) + \frac{R+2\epsilon n}{tr_gg'} - \frac{\Delta F}{tr_gg'}
\end{aligned} 
\end{equation}
Let $\tilde{w}=\log(tr_g g')-A\phi$. Then $w=e^{-A\phi}(tr_gg') = \exp(\tilde{w})$. We have
\begin{equation}
\Delta' w = \Delta e^{\tilde{w}}
= e^{\tilde{w}} \Delta' \tilde{w} + e^{\tilde{w}}|\nabla' \tilde{w}|^2_\phi 
\geq  e^{\tilde{w}}\Delta' \tilde{w}
\end{equation}
Namely,
\begin{equation}\label{equ7}
\begin{aligned}
\Delta' w 
\geq &\ e^{-A\phi}(tr_gg') \Delta'(\log(tr_g g')-A\phi) \\
= &\ e^{-A\phi}\left[(tr_gg')(tr_{g'}g)-(2nA+\epsilon)(tr_gg')+(R+2\epsilon n)\right] - e^{-A\phi}\Delta F
\end{aligned}
\end{equation}
It is well known that 
\begin{equation}
tr_{g'}g \geq e^{-\frac{F}{n-1}}(tr_g g')^{\frac{1}{n-1}}
\end{equation}
Hence,
\begin{equation}
\begin{aligned}
&\ e^{-A\phi}\left[(tr_gg')(tr_{g'}g)-(2nA+\epsilon)(tr_gg')+(R+2\epsilon n)\right] \\
\geq\ &\ e^{-A\phi}\left[e^{-\frac{F}{n-1}}(tr_gg')^\frac{n}{n-1}-(2nA+\epsilon)(tr_gg')+(R+2\epsilon n)\right] \\
\geq\ &\ e^\frac{A\phi-f}{n-1}w^\frac{n}{n-1}-(2nA+\epsilon)w +(R+2\epsilon n)e^{-A\phi}\\
\geq\ &\ 2\theta w^{\frac{n}{n-1}}-(2nA+1)w -C_0
\end{aligned}
\end{equation}
where $\theta = \frac{1}{2}\exp(-\frac{A|\phi|+|F|}{n-1})$ and $C_0= (|R|+2n)e^{A|\phi|}$. On the other hand, we have 
\begin{equation}
(2nA+1)w\leq \theta w^\frac{n}{n-1}+C_\theta
\end{equation} 
Hence,
\begin{equation}\label{equ8}
e^{-A\phi}\left[(tr_gg')(tr_{g'}g)-(2nA+\epsilon)(tr_gg')+(R+2\epsilon n)\right]\geq \theta w^\frac{n}{n-1} - C_1
\end{equation}
where $C_1=C_0+C_\theta$. Combine (\ref{equ7}) and (\ref{equ8}), we have
\begin{equation}
\Delta' w \geq \theta w^\frac{n}{n-1} - C_1 - e^{-A\phi}\Delta F
\end{equation}
We finish the proof.
\end{proof}
\vspace{1em}

Now we start to prove Theorem \ref{thm3}.

\begin{proof}[Proof of theorem \ref{thm3}]
We compute, for $p>0$,

\begin{equation}\label{equ9}
\begin{aligned}
\int_M |\nabla w^p|^2dV 
\leq &\ \int_M (tr_gg')|\nabla' w^p|^2_\phi e^{-(F+\epsilon\phi)}dV' \\
= &\ \int_M e^{(A-\epsilon)\phi-F}w|\nabla' w^p|^2_\phi dV' \\
\leq &\ C_2\int_M w|\nabla' w^p|^2_\phi dV'
\end{aligned}
\end{equation}

where $C_2=\exp((A+1)|\phi|+|F|)$. Continue the computation, we have
\begin{equation}\label{equ10}
\begin{aligned}
\int_M w|\nabla' w^p|^2_\phi dV' &= \frac{p}{2} \int_M \nabla'(w^{2p}\nabla' w) dV' -\frac{p}{2}\int_M w^{2p}\Delta' w dV' \\
&= -\frac{p}{2} \int_M w^{2p}\Delta' w dV'
\end{aligned}
\end{equation}
The integral $\int \nabla'(w^{2p}\nabla' w)dV'$ in the right hand side vanishes because of Gaffney-Stokes. Combining (\ref{equ9}) and (\ref{equ10}), we have
\begin{equation}
\begin{aligned}
\int_M |\nabla w^p|^2 dV \leq &\ \frac{1}{2}C_2p\int w^{2p}(-\Delta' w) dV' \\
= &\ \frac{1}{2}C_2p\int_M w^{2p}(-\Delta' w) e^{F+\epsilon\phi}dV \\
\leq &\ C_3p\int_M w^{2p}(-\Delta' w)dV
\end{aligned}
\end{equation}
where $C_3=\frac{1}{2}C_2\exp(|F|+|\phi|)$.
By Lemma \ref{lem2}, we have
\begin{equation}
\begin{aligned}
\int_M |\nabla w^p|^2 dV 
\leq &\ C_3p\left\{-\int_M w^{2p}(\theta w^\frac{n}{n-1}-C_1) dV + \int_M w^{2p}e^{-A\phi}\Delta F dV\right\}\\
\leq &\ C_3p\left\{-\int_M w^{2p}(\theta w^\frac{n}{n-1}-C_1)dV + \int_M 2e^{-A\phi}w^p|\nabla w^p||\nabla F|dV \right.\\
&\qquad\quad \left. +\int_M Ae^{-A\phi}w^{2p}|\nabla \phi||\nabla F|dV\right\} \\
\leq &\ -C_3p\int_M w^{2p}(\theta w^\frac{n}{n-1}-C_1)dV + C_4p\int_M w^p|\nabla w^p||\nabla F|dV \\
&\qquad\quad + C_5p\int_M w^{2p}|\nabla F|dV \\
\end{aligned}
\end{equation}
where $C_4=2C_3\exp(A|\phi|)$ and $C_5=C_3A\exp(A|\phi|)|\nabla \phi|$.
By H\"{o}lder inequality,
\begin{gather}
C_4p\int_M |\nabla w^p|w^p|\nabla F|dV \leq \frac{1}{2}\int_M |\nabla w^p|^2 dV + \frac{1}{2}(C_4p)^2\int_M w^{2p}|\nabla F|^2 dV\\
\int_M w^{2p}|\nabla F|dV \leq \frac{1}{2}\int_M w^{2p}(|\nabla F|^2 +1)
\end{gather}
Hence,
\begin{equation}
\int_M |\nabla w^p|^2 dV \leq -2C_3p\int_M w^{2p}(\theta w^\frac{n}{n-1}-C_6)dV + (C_4^2p^2+C_5p)\int_M w^{2p}|\nabla F|^2 dV
\end{equation}
At this stage, we have
\begin{equation}\label{equ13}
\int_M |\nabla w^p|^2dV\leq -C_7p\int_M w^{2p}(\theta w^\frac{n}{n-1}-C_6)dV + C_8p^2\int_M w^{2p}|\nabla F|^2dV
\end{equation}
We compute 
\begin{equation}\label{equ11}
\int_M |\nabla (w^p\rho^{-\frac{1}{2}})|^2 \rho dV \leq \int_M |\nabla w^p|^2 dV +\int_M w^{2p}\left(\frac{|\nabla \rho|}{\rho}\right)^2 dV
\end{equation}
By the weighted Sobolev inequality (Lemma \ref{sob_inequ}), we have
\begin{equation}\label{equ12}
\left(\int_M (w^p\rho^{-\frac{1}{2}})^{\frac{4n}{2n-2}}\rho dV\right)^{\frac{2n-2}{4n}} \leq C_S\left(\int_M |\nabla (w^p\rho^{-\frac{1}{2}})|^2\rho dV + \int_M w^{2p} dV\right)^{\frac{1}{2}}
\end{equation}
Combining (\ref{equ11}) and (\ref{equ12}), we get
\begin{equation}
\begin{aligned}
\left(\int_M w^{2p\frac{n}{n-1}}\rho^{-\frac{1}{n-1}}dV\right)^\frac{n-1}{2n} 
\leq &\ C_S\left\{\int_M |\nabla  w^p|^2 dV +
\int_M w^{2p}\left(\left(\frac{|\nabla \rho|}{\rho}\right)^2+1\right)dV\right\}^\frac{1}{2}\\
\leq &\ C_S\left\{ \int_M |\nabla w^p|^2 dV + C_9\int w^{2p} dV \right\}^\frac{1}{2} 
\end{aligned}
\end{equation}
Combining with (\ref{equ13}), and let $d\nu=\rho^{-\frac{1}{n-1}}dV$, we have
\begin{equation}
\begin{aligned}
\left(\int_M w^{2p\frac{n}{n-1}}d\nu\right)^\frac{n-1}{n} 
\leq &\ {C_S}^2\left\{ -C_7p \int_M w^{2p}(\theta w^\frac{n}{n-1}-C_{10})dV \right.  \\
&\ + \left. C_8p^2\int_M w^{2p}|\nabla F|^2 \rho^{\frac{1}{n-1}}d\nu\right\}
\end{aligned}
\end{equation}
Let $K=(C_{10}\theta^{-1})^{\frac{n-1}{n}}$. Note that when $w> K$, $C_{10}-\theta w^\frac{n}{n-1}<0$. Hence,
\begin{equation}
\int_M w^{2p}(C_{10}-\theta w^\frac{n}{n-1})dV 
\leq \int_{\{w\leq K\}} C_{10}w^{2p} dV
\leq C_{10}K^{2p}Vol(M)
\end{equation}
On the other hand, by H\"{o}lder inequality, we have
\begin{equation}
\begin{aligned}
\int_M w^{2p}|\nabla F|\rho^\frac{1}{n-1}d\nu
\leq&\ \left(\int_M w^{2pq_0}d\nu\right)^\frac{1}{q_0} \left(\int_M |\nabla F|^{p_0}\rho^\frac{p_0}{2n-2}d\nu\right)^\frac{2}{p_0} \\
=&\ \left(\int_M w^{2pq_0}d\nu\right)^\frac{1}{q_0} \left(\int_M |\nabla F|^{p_0}\rho^\frac{p_0-2}{2n-2}dV\right)^\frac{2}{p_0}
\end{aligned}
\end{equation}
where $\frac{1}{q_0}+\frac{2}{p_0}=1$. Since $p_0>2n$, we have $q_0< \frac{n}{n-1}$.
Hence,
\begin{equation}
\left(\int_M w^{2p\frac{n}{n-1}} d\nu \right)^\frac{n-1}{n} \leq {C_S}^2 \left\{C_{11}qK^{2p}+ C_{12}p^2\left(\int_M w^{2pq_0}d\nu\right)^\frac{1}{q_0}\right\}
\end{equation}
Let $w=Kv$. Then we have
\begin{equation}
||v||_{L^{2p\frac{n}{n-1}}(d\nu)}^p \leq C_{13}p^2(||v||^p_{L^{2pq_0}(d\nu)}+1)
\end{equation}
A similar iteration argument as in the end of proof of section 4 shows that
\begin{gather}
||v||_{L^\infty} \leq C||v||_{L^1(d\nu)}\leq C\int v dV,
\end{gather}
which gives us
\begin{gather}
||w||_{L^\infty} \leq C\int w dV\leq C\int (tr_gg') dV \leq C.
\end{gather}
We finish the proof.
\end{proof}

\section{H\"{o}lder estimate of second order and $W^{3,p_0}$ estimate}

\subsection{H\"{o}lder estimate of the second order}
The H\"{o}lder estimates of second order derivatives are studied for uniform elliptic operators when the right hand side has weaker regularity. For Monge-Amp\'{e}re equation, Blocki proved that the H\"{o}lder esitmates hold when $F$ is Lipschitz and $\Delta\phi$ is bounded \cite{blocki}. Chen-He extended Blocki's result to the case when $F$ is only $W^{1,p_0},\, p_0>2n$ \cite{Chen-He}. The estimate can be made local. We collect Chen-He's result in the following lemma.

\begin{lemma}[\cite{Chen-He}, Lemma 4.1]\label{lem4}
Let $v$ be a $C^4$-psh function in an open $\Omega\subset \mathbb{C}^n$ such that $\det(v_{i\bar{j}})=f$. Assume that for some positive $\lambda,\Lambda$ and $K$, we have 
$$0<\lambda\leq \Delta v\leq \Lambda, \quad ||v||_{L^\infty}\leq K,\quad ||f||_{W^{1,p_0}}\leq K,$$
then for any $\Omega'\subset\subset \Omega$, there exists some $\alpha=\alpha(\Omega,\Omega',\lambda,\Lambda,K)$, $0<\alpha<1$, such that
$$||v||_{C^{2,\alpha}(\Omega')}\leq C(\Omega,\Omega',\lambda,\Lambda,K).$$
\end{lemma}

We could not apply the local H\"{o}lder estimate directly as Chen-He did in the compact case, since $(M,\omega)$ is complete noncompact and can not be covered by finitely many local coordinates charts. Instead, we consider to pull back the Monge-Amper\'{e} equation to the quasi-coordinate charts and obtain the uniform $C^{2,\alpha}_{qc}$ norm bound. In particular, we have

\begin{proposition}
Let $F\in C^\infty_c(M,g)$ satisfies $I(F,p_0)<\infty$, and $\phi_\epsilon$ be the solution of equation (\ref{equ1}). Then we have 
\[
||\phi_\epsilon||_{C^{2,\alpha}_{qc}(M,g)} \leq C
\]
for some $\alpha,\, 0 <\alpha < 1$ and postive constant $C$ depending on $(I(F,p_0),\, p_0,\, n,\, g)$.
\end{proposition}

\begin{proof}
For simplicity of notations, we will drop the subscript $\epsilon$ from $\phi_\epsilon$. 
For a quasi-coordinate chart $\Phi_\delta: \frac{3}{4}\mathcal{P}_k\rightarrow U\backslash D$, we pull back the equation (\ref{equ1}) to quasi-coordinate chart. Let $\tilde{g}={\Phi_\delta}^*g,\, \tilde{\phi}={\Phi_\delta}^*\phi$, and $\tilde{F}={\Phi_\delta}^*F$, we have
\begin{equation}\label{equ18}
\det\left(\tilde{g}_{i\bar{j}}+\tilde{\phi}_{i\bar{j}}\right) = \exp(\tilde{F}+\epsilon\tilde{\phi})\det\left(\tilde{g}_{i\bar{j}}\right).
\end{equation}
on the polydisc $\frac{3}{4}\mathcal{P}_k\subset \mathbb{C}^n$. Note that we have $\tilde{g}$ quasi-isometric to the Euclidean metric $g_0$ on $\polydisc$. Find some local potential $G_0$ on $\mathcal{P}$ such that $\tilde{g}_{i\bar{j}}=(G_0)_{i\bar{j}}$ and rewrite (\ref{equ18}) in $\frac{3}{4}\mathcal{P}_k$ as 
\begin{equation}
\det(\vartheta_{i\bar{j}})= f
\end{equation}
where $\vartheta = G_0+\tilde{\phi}$ and $f=\exp(\tilde{F}+\epsilon\tilde{\phi})\det((g_0)_{i\bar{j}})$. Since $||\phi||_{L^\infty}$ and $||\Delta \phi||_{L^\infty}$ are uniformly bounded, we can assume that $0< \lambda \leq \Delta \vartheta \leq \Lambda$ and $||\vartheta||_{L^\infty}\leq C$. We claim that $f
\in W^{1,p_0}(\frac{3}{4}\mathcal{P}_k)$ and $||f||_{W^{1,p_0}(\frac{3}{4}\mathcal{P}_k)}\leq C$. To see this, first note that $||f||_{L^\infty}\leq C$, it suffices to show that $||\nabla_{g_0} \tilde{F}||_{L^p(\frac{3}{4}\mathcal{P}_k)}\leq C$. 
As indicated in the proof of Lemma \ref{lem3}, we have
\begin{equation}
\begin{aligned}
\int_{\polydisc} |\nabla_{\tilde{g}} \tilde{F}|^{p_0} dV_{\tilde{g}}
\leq&\ C\int_{\frac{3}{4}\mathcal{P}_k} |\nabla_{g_0} \tilde{F}|^{p_0}dV_{g_0}  \\
\leq&\ C\int_{\frac{3}{4}\mathcal{P}_k} ({\Phi_\delta}^*\rho)^{-1} {\Phi_\delta}^*(|\nabla F|^{p_0}\rho)dV_{g_0} \\
\leq&\ C\int_{U\backslash D} |\nabla F|^{p_0}\rho dV \\
\leq&\ I(F,p_0).\\
\end{aligned}
\end{equation}
Hence, by Lemma \ref{lem4}, we have 
\[
||\vartheta||_{C^{2,\alpha}(\frac{3}{4}\mathcal{P}_k)}\leq C
\]
The $\alpha$ and $C$ are uniform with respect to the quasi-coordinate charts. Hence, take the supreme over all the quasi-coordinate charts, we get 
\[
||\phi||_{C^{2,\alpha}_{qc}}\leq C.
\]
\end{proof}

\subsection{$W^{3,p_0}$ estimate and solve the equation}
To obtain $W^{3,p_0}$ estimate, we can localize the estimates in the quasi-coordinate charts as follows. Let $\partial$ be an arbitrary first order differential operator in the quasi-coordinate chart $\frac{3}{4}\mathcal{P}_k$. Once the H\"{o}lder estimate of second order is proved, we compute in the quasi-coordinate chart
\begin{equation}
\Delta_{\tilde{g}} \partial \tilde{\phi} = \partial(\tilde{F}+\epsilon\tilde{\phi}+\log\det(\tilde{g}_{i\bar{j}}))  - \tilde{g}^{i\bar{j}}_{\tilde{\phi}} \partial\tilde{g}_{i\bar{j}}.
\end{equation}
Note that we already have $||\tilde{\phi}||_{C^{2,\alpha}(\polydisc)}$ and $||\tilde{F}||_{W^{1,p_0}(\polydisc)}$ bounded, hence the $L^{p_0}$ norm of right hand side bounded. It then follows from $L^p$ theory, for example see \cite{GT} Chapter 9, that
\begin{equation}
||\partial\tilde{\phi}||_{W^{2,p_0}(\frac{1}{2}\mathcal{P}_k)}
\leq C.
\end{equation}
Namely,
\begin{equation}
||{\Phi_\delta}^*\phi||_{W^{3,p_0}(\frac{1}{2}\mathcal{P}_k)} 
\leq C.
\end{equation}
By Lemma \ref{lem3}, we have
\begin{equation}
\begin{aligned}
||\phi||^{p_0}_{W^{3,p_0}(U\backslash D)} 
\leq&\ c \sum_{\ell} A_{\delta_\ell} ||{\Phi_\delta}^*\phi||^{p_0}_{W^{3,p_0}(\frac{1}{2}\mathcal{P}_k)} \\
\leq&\ C \sum_{\ell} A_{\delta_\ell} \int_{\polydisc} 1 dV_{g_0}\\
\leq&\ C \int_{U\backslash D} 1 dV \\
\leq&\ CVol(M).
\end{aligned}
\end{equation}
Cover $D$ by finitely many such $U$'s and cover the complement of the union of these $U$'s by a finite number of unit balls. Collect the inequalites on each of them, we get the following proposition.

\begin{proposition} \label{prop2}
Let $F\in C^\infty_c(M,g)$ satisfies $I(F,p_0)<\infty$, and $\phi_\epsilon$ be the solution of equation (\ref{equ1}). Then we have 
\begin{equation}
||\phi_\epsilon||_{W^{3,p_0}(M)}\leq C,
\end{equation}
where $C$ depends only on $(I(F,p_0),p_0,n,\omega)$.
\end{proposition}

The following lemma show us that $F$ can be approximated by $C^\infty_c$ functions in the weighted Sobolev spaces.
\begin{lemma}\label{app_lem}
Suppose $F$ sastifies $I(F,p_0)\leq \infty$, then there is a sequence of $F_k\in C^\infty_c$ such that $I(F-F_k,p_0)\rightarrow 0$ as $k\rightarrow 0$. In particular, $F_k\rightarrow F$ in $W^{1,p_0}(M,g)$.
\end{lemma}
\begin{proof}
We can assume that $F$ is smooth. The Ricci curvature of $(M,g)$ is bounded from below. Let $r=r(x)$ denote the distance function to some fixed point. By a theorem of Yau (\cite{Scheon-Yau}, theorem 4.2), there is a proper $C^\infty(M)$ function $d$ such that $d\geq Cr$ and $|\nabla d|\leq C$, for some constant $C$. Let $\chi :[0,\infty)\rightarrow \mathbb{R}$ be a cut-off function such that: (i) $\chi(t) \equiv 1$  for $t \leq 1$, $\chi(t) \equiv 0$ for $t\geq 2$ and $0 \leq \chi \leq 1$; (ii) $|\chi'(t)|< 2$. Let $F_k(x)= \chi\left({d(x)}/{k}\right)F(x)$. Then $F_k\in C^\infty_c(M,g)$. It remains to show 
$I(F-F_k, p_0)\rightarrow 0$. To see this, note that $\nabla F_k = \chi\left({d}/{k}\right)\nabla F + {F}\chi'\left({d}/{k}\right)\nabla d/k$. Hence,
\[
\begin{aligned}
&\ \int ( |F-F_k|^{p_0} + |\nabla F - \nabla F_k|^{p_0} ) \rho^{\frac{p_0-2}{2n-2}} dV \\
\leq &\ \int \left(1-\chi\left({d}/{k}\right)\right)|F|^{p_0} + \left( (1-\chi({d}/{k}))|\nabla F|+Ck^{-1}|F|\right)^{p_0} \rho^{\frac{p_0-2}{2n-2}} dV \\
\leq &\ C\int_{\{d\geq k\}} (|F|^{p_0}+|\nabla F|^{p_0})\rho^{\frac{p_0-2}{2n-2}} dV + Ck^{-1}\int |F|^{p_0} \rho^{\frac{p_0-2}{2n-2}}dV
\end{aligned}
\]
The RHS goes to $0$ as $k\rightarrow \infty$.
\end{proof}

Finally, we proof the main theorem. 
\begin{proof}[Proof of theorem \ref{thm1}]
Suppose $I(F,p_0)\leq K$ for some constant $K$. Let $F_k$ be a sequence of smooth functions with compact support such that $I(F_k,p_0)\rightarrow I(F,p_0)$; in particular, we can assume $I(F_k,p_0)\leq K+1$ for any $k$. For each $\epsilon$ and $k$, there is a smooth solution $\phi_{\epsilon,k}$ which solves the perturbed equation
\begin{equation}
(\omega + \ddbar \phi_{\epsilon,k})^n = e^{F_k+\epsilon \phi_{\epsilon,k}}\omega^n
\end{equation}
such that $(g_{i\bar{j}}+(\phi_{\epsilon,k})_{i\bar{j}})>0$. By Proposition \ref{prop2} we have 
\begin{equation}
||\phi_{\epsilon,k}||_{W^{3,p_0}(M,g)}\leq C(K,p_0,n,g)
\end{equation}
There is a subsequence of $(\phi_{\epsilon,k})$ that converges to some $\phi\in W^{3,p_0}(M,g)$ such that $\omega + \ddbar\phi>0$ defines a $W^{1,p_0}$ (and $C^{\alpha}_{qc},\, \alpha= 1-2n/p_0$) K\"{a}hler metric. 
\end{proof}


\vspace{2em}

\footnotesize{\textsc{Department of Mathematics, Stony Brook University, Stony Brook, NY, 11794}}

\footnotesize{Email: \texttt{fangyu.zou@stonybrook.edu}}
\end{document}